\documentclass[12pt, reqno]{amsart}
\usepackage{amssymb,latexsym,amsmath,amscd,amsthm,graphicx, color}
\usepackage[all]{xy}
\usepackage{pgf,tikz}
\usepackage{mathrsfs}
\usetikzlibrary{arrows}
\usepackage{pgfplots}
\pgfplotsset{compat=1.15}
\usepackage[left=0.6 in, top=0.6 in, right=0.6 in, bottom=0.3 in]{geometry}
\usepackage{changepage}
\allowdisplaybreaks
\makeatletter
\newcommand{\setword}[2]{%
	\phantomsection
	#1\def\@currentlabel{\unexpanded{#1}}\label{#2}%
}
\makeatother

\usepackage[linkcolor=blue, urlcolor=blue, citecolor=blue,
colorlinks, bookmarks]{hyperref}

\definecolor{uuuuuu}{rgb}{0.26666666666666666,0.26666666666666666,0.26666666666666666}
\definecolor{xdxdff}{rgb}{0.49019607843137253,0.49019607843137253,1.}
\definecolor{ffqqqq}{rgb}{1.,0.,0.}
\definecolor{ffqqqq}{rgb}{1.,0.,0.}
\definecolor{ffxfqq}{rgb}{1.,0.4980392156862745,0.}

\definecolor{uuuuuu}{rgb}{0.26666666666666666,0.26666666666666666,0.26666666666666666}
\definecolor{qqwuqq}{rgb}{0.,0.39215686274509803,0.}
\definecolor{zzttqq}{rgb}{0.6,0.2,0.}
\definecolor{xdxdff}{rgb}{0.49019607843137253,0.49019607843137253,1.}
\definecolor{qqqqff}{rgb}{0.,0.,1.}
\definecolor{cqcqcq}{rgb}{0.7529411764705882,0.7529411764705882,0.7529411764705882}
\definecolor{sqsqsq}{rgb}{0.12549019607843137,0.12549019607843137,0.12549019607843137}
\definecolor{uuuuuu}{rgb}{0.26666666666666666,0.26666666666666666,0.26666666666666666}
\definecolor{ffqqqq}{rgb}{1,0,0}
\definecolor{xdxdff}{rgb}{0.49019607843137253,0.49019607843137253,1}

\definecolor{yqqqyq}{rgb}{0.5019607843137255,0,0.5019607843137255}
\definecolor{qqqqff}{rgb}{0,0,1}
\definecolor{ffqqqq}{rgb}{1,0,0}
\definecolor{ffqqff}{rgb}{1,0,1}

\theoremstyle{plain}

\newtheorem{theorem}[subsection]{Theorem}

\newtheorem{lemma}[subsection]{Lemma}

\newtheorem{prop}[subsection]{Proposition}

\theoremstyle{definition}

\newtheorem{exam}[subsection]{Example}

\newtheorem{remark}[subsection]{Remark}

\newtheorem{open}[subsection]{Open}

\newcommand{\sci}{\subset}

\newcommand{\set}[1]{\{#1\}}

\newcommand{\ga}{\alpha}
\newcommand{\gb}{\beta}

\newcommand{\tbf}{\textbf}
\newcommand{\tit}{\textit}

\newcommand{\C}[1]{\mathcal{#1}}
\newcommand{\D}[1]{\mathbb{#1}}

\newcommand{\te}{\text}

\newcommand{\nd}{\noindent}

\begin{document}
\nd To appear, Contemporary Mathematics of American Mathematical Society
\title{Optimal quantization for nonuniform discrete distributions}

\address{School of Mathematical and Statistical Sciences\\
University of Texas Rio Grande Valley\\
1201 West University Drive\\
Edinburg, TX 78539-2999, USA.}

\email{\{$^1$rljcabasag07, $^2$samir.huq01, $^3$ericmendoza128\}@gmail.com}
\email{$^4$mrinal.roychowdhury@utrgv.edu}

\author{$^1$Russel Cabasag}
 \author{$^2$Samir Huq}
 \author{$^3$Eric Mendoza}
  \author{$^4$Mrinal Kanti Roychowdhury}



\date{}
\maketitle

\pagestyle{myheadings}\markboth{R. Cabasag, S. Huq, E. Mendoza, M.K. Roychowdhury}
{Optimal quantization for nonuniform discrete distributions}

\begin{abstract}
This paper explores the process of optimal quantization for several types of discrete probability distributions. Quantization is a technique used to approximate a complex distribution with a smaller set of representative points, which is important in fields such as data compression and signal processing. We begin by examining two specific nonuniform distributions over a finite set of values and identify the best representative points for different levels of approximation. We then extend our analysis to two infinite discrete distributions: one supported on the reciprocals of natural numbers and another on the natural numbers themselves. For these distributions, we compute the optimal sets of representatives and assess how well they approximate the original distributions. Finally, we address the reverse problem—determining the underlying distribution when the optimal sets are known. Our results provide both theoretical insights and computational techniques useful in information theory and data analysis.
\end{abstract}

\section{Introduction}

Quantization is the process of converting a continuous analog signal into a digital signal of $k$ discrete levels, or converting a digital signal of $n$ levels into another digital signal of $k$ levels, where $k < n$.
It is essential when analog quantities are represented, processed, stored, or transmitted by a digital system, or when data compression is required. It is a classic and still very active research topic in source coding and information theory. It has broad applications in engineering and technology. In fact, it is useful in data compression and cluster analysis. For some details one can see  \cite{DFG, GG, GL1, GN,  P, Z1, Z2}. Recently, the quantization theory has been extended to constrained quantization and conditional quantization (see \cite{PR1,PR2, PR3}). Following the introduction of constrained quantization, the theory of quantization is now categorized into two main branches: constrained quantization and unconstrained quantization.
To gain a quick understanding and practical grasp of research in constrained and unconstrained quantization, the paper \cite{PR1} serves as a helpful starting point. For a more in-depth treatment of unconstrained quantization, one may refer to \cite{GL2}.

Let $\D R^d$ denote the $d$-dimensional Euclidean space equipped with a metric $\|\cdot\|$ compatible with the Euclidean topology. Let $P$ be a Borel probability measure on $\D R^d$ and $\ga$ be a locally finite subset of $\D R^d$, i.e., intersection of $\ga$ with any bounded subset of $\D R^d$ is finite. This implies that $\ga$ is countable and closed. 
Then, $\int \min_{a \in \ga} \|x-a\|^2 dP(x)$ is often referred to as the \tit{cost,} or \tit{distortion error} for $\ga$ with respect to the probability measure $P$, and is denoted by $V(P; \ga)$.  Write
$\C D_n:=\set{\ga \sci \D R^d : 1\leq \te{card}(\ga)\leq n}$.  Then, $\inf\set{V(P; \ga) : \ga \in \C D_n}$ is called the \tit{$n$th quantization error} for the probability measure $P$, and is denoted by $V_n:=V_n(P)$. A set $\ga$ for which the infimum occurs and contains no more than $n$ elements is called an \tit{optimal set of $n$-points}. It is known that for a Borel probability measure $P$,  if its support contains at least $n$ elements and $\int \| x\|^2 dP(x)$ is finite, then an optimal set of $n$-points always has exactly $n$-elements (see \cite{ GL1, GL2, GKL, AW}). 
Recently, optimal quantization for different uniform distributions have been investigated by several authors, for example, see \cite{DR, PRRSS, R, RR, RS}.

Given a finite set $\ga\sci \D R^d$, the Voronoi region generated by $a\in \ga$ is defined by
\[M(a|\ga)=\set{x \in \D R^d : \|x-a\|=\min_{b \in \ga}\|x-b\|},\]
i.e., the Voronoi region generated by $a\in \ga$ is the set of all elements in $\D R^d$ which are nearest to $a$, and the set $\set{M(a|\ga) : a \in \ga}$ is called the \tit{Voronoi diagram} or \tit{Voronoi tessellation} of $\D R^d$ with respect to $\ga$.

Let us now state the following proposition (see \cite{ GG, GL2}).
\begin{prop} \label{prop0}
Let $\ga$ be an optimal set of $n$-points for $P$, and $a\in \ga$. Then,

$(i)$ $P(M(a|\ga))>0$, $(ii)$ $ P(\partial M(a|\ga))=0$, $(iii)$ $a=E(X : X \in M(a|\ga))$,
where $M(a|\ga)$ is the Voronoi region of $a\in \ga, $ i.e.,  $M(a|\ga)$ is the set of all elements $x$ in $\D R^d$ which are closest to $a$ among all the elements in $\ga$, and $\partial M(a|\ga)$ represents the boundary of the Voronoi region $M(a|\ga)$.
\end{prop}
By the above proposition, we see that in unconstrained quantization, the elements in an optimal set of $n$-points are the conditional expectations in their own Voronoi regions. Because of this fact, in unconstrained quantization, an optimal set of $n$-points is termed an optimal set of 
$n$-means.
 
\subsection{Delineation} In this paper, we investigate the optimal quantization for finite, and infinite discrete distributions. Section~\ref{sec0} contains the basic preliminaries. In Section~\ref{sec1}, we calculate the optimal sets of $n$-means and the $n$th quantization errors for all $1\leq n\leq 6$ for two nonuniform discrete distributions with support $\set{1, 2, 3, 4, 5, 6}$ associated with two different probability vectors. In Section~\ref{sec2} first, for a probability distribution $P$ with support $\set{\frac 1n : n\in \D N}$ associated with a mass function $f$, given by $f(x)=\frac 1 {2^k}$ if $x=\frac 1 k$ for $k\in \D N$, and zero otherwise, we determine the optimal sets of $n$-means and the $n$th quantization errors for all positive integers up to $n=300$. Then, for a probability distribution $P$ with support the set $\D N$ of natural number associated with a mass function $f$, given by $f(x)=\frac 1 {2^k}$ if $x=k$ for $k\in \D N$, and zero otherwise, we determine the optimal sets of $n$-means and the $n$th quantization errors for all positive integers $n$. In Section~\ref{sec3}, we discuss for a discrete distribution, if the optimal sets are given, how to obtain the probability distributions.

\section{Basic Preliminaries} \label{sec0} 
Let $P$ be a Borel probability measure on $\D R^d$, and $U$ be the largest open subset of $\D R^d$ for which $P(U)=0$. Then, $\D R^d \setminus U$ is called the support of $P$, and is denoted by supp$(P)$. Let $P$ be a uniform distribution defined on the set $\set{1, 2, 3, 4, 5, 6}$. Then, the random variable $X$ associated with the probability distribution is a discrete random variable with probability mass function $f$ given by
 \[f(x)=P(X=x)=\frac 16, \te{ where } x\in \set{1, 2, 3, 4, 5, 6}.\]
It is not difficult to show that if $\ga_n$ is an optimal set of $n$-means for $P$, then
\begin{align*}
\ga_1=\set{3.5}, \  \ga_2& =\set{2, 5}, \ \ga_3=\set{1.5, 3.5, 5.5}, \  \ga_4=\set{1.5, 3.5, 5, 6}, \\
 \ga_5&=\set{1.5, 3, 4, 5, 6}, \te{ and } \ga_6=\te{supp}(P).
\end{align*}
\begin{remark}
Optimal sets are not unique. For example, in the above, the set $\ga_5$ can be any one of the following sets:
\[\set{1.5, 3, 4, 5, 6}, \, \set{1, 2.5, 4, 5, 6}, \, \set{1, 2, 3.5, 5, 6}, \, \set{1, 2, 3, 4.5, 6}, \, \set{1, 2, 3, 4, 5.5}.\]
\end{remark}
In the following sections we give our main results.

\section{Optimal quantization for nonuniform discrete distributions} \label{sec1}

Let $P$ be a nonuniform distribution defined on the set $\set{1, 2, 3, 4, 5, 6}$ associated with a probability vector of the form $(x, (1 - x) x, (1 - x)^2 x, (1 - x)^3 x, (1 - x)^4 x, (1 - x)^5)$, i.e., the probability mass function $f$ is given by
\[f(j)=P(X=j)=\left\{\begin{array} {cc} 
 \vspace{0.05 in} 
x \te{ if } j=1,\\  
 \vspace{0.05 in} 
(1-x)^{j-1}x \te{ if } j\in \set{ 2, 3, 4, 5}, \\  
 \vspace{0.05 in} 
(1-x)^{5} \te { if } j=6,\\
0 \te{ otherwise, }
\end{array}
\right.
 \]
 where $0<x<1$.  Let $X$ be a random variable with probability distribution $P$. Then,
\begin{align*}
E(X)& =1 x+\sum _{j=2}^5 j x (1-x)^{j-1}+6 (1-x)^5=-x^5+6 x^4-15 x^3+20 x^2-15 x+6,
\end{align*}
and so, the optimal set of one-mean is the set $\set{-x^5+6 x^4-15 x^3+20 x^2-15 x+6}$ with quantization error the variance $V$ of the random variable $X$, where
\begin{align*}
V=V_1&=E\|X-E(X)\|^2=\sum_{j=1}^6 f(j)\Big(j-E(X)\Big)^2\\
&= -x (-55 + 275 x - 627 x^2 + 858 x^3 - 781 x^4 + 495 x^5 - 220 x^6 +
   66 x^7 - 12 x^8 + x^9).
   \end{align*}
 Notice that supp$(P)=\set{1, 2, 3, 4, 5, 6}$, and it forms the optimal set of six-means for the probability distribution $P$ for any $0<x<1$.
 For $i, j\in \set{1, 2, \cdots, 6}$ with $i\leq j$, write \[a[i, j]:=E(X : X \in \set{i, i+1, \cdots, j})=\frac {\sum_{x=i}^j xf(x)}{\sum_{x=i}^j  f(x)}.\]
 In the following two subsections, we determine the optimal sets of $n$-means and the $n$th quantization errors for all $2\leq n\leq 5$ for the probability distribution $P$: one for $x=\frac 12$, and one for $x=\frac 7{10}$.

\subsection{Optimal quantization for the probability distribution $P$ with $x=\frac 12$.}
In this case, the probability mass function $f$ is given by
\[f(j)=P(X=j)=\left\{\begin{array} {ll}\vspace{0.05 in}
\frac 1{2^j} \te{ if } j\in \set{1, 2, 3, 4, 5}, \\
 \vspace{0.05in} 
\frac 1 {2^5} \te { if } j=6,\\
0 \te{ otherwise. }
\end{array}
\right.
 \]
 We now give the following propositions.
\begin{prop} \label{prop12}
The optimal set of two-means is given by $\set{a[1,2], a[3, 6]}$ with quantization error $V_2=\frac{341}{768}$.
\end{prop}

\begin{proof} Notice that $a[1,2]=\frac 43$, and $a[3,6]=\frac {31} 8$.
Let us consider the set $\gb:=\set{\frac 43, \frac {31}8}$. Since $2<\frac 12(\frac 43+\frac {31}8)=2.60417<3$,
the distortion error due to the set $\gb$ is given by
\begin{align*} \sum_{j=1}^6 f(j)\min_{a\in \gb} (j-a)^2=\sum _{j=1}^2 f(j)\Big(j-\frac{4}{3}\Big)^2+\sum _{j=3}^6f(j)\Big(j-\frac{31}{8}\Big)^2=\frac{341}{768}.
\end{align*}
Since $V_2$ is the quantization error for two-means, we have $V_2\leq \frac{341}{768}=0.44401$.
Let $\ga:=\set{a_1, a_2}$ be an optimal set of two-means. Without any loss of generality, we can assume that $1\leq a_1<a_2\leq 6$. Notice that the Voronoi region of $a_1$ must contain $1$. Suppose that the Voronoi region of $a_1$ contains the set $\set{1,2,3}$. Then, as $a[1,3]=\frac{11}{7}$, we have
\[V_2\geq \sum _{j=1}^3f(j)\Big(j-\frac {11}{7}\Big)^2=\frac{13}{28}=0.464286>V_2,\]
which gives a contradiction. Hence, we can assume that the Voronoi region of $a_1$ does not contain $\set{1,2,3}$. Next, suppose that the Voronoi region of $a_1$ contains only the element $1$. Then, the Voronoi region of $a_2$ contains all the remaining elements, and so
\[a_2=a[2, 6]=\frac{47}{16}\]
implying
\[V_2=\sum _{j=2}^6f(j)\Big(j-\frac{47}{16}\Big)^2=\frac{367}{512}=0.716797>V_2,\]
which yields a contradiction.
Hence, we can assume that the Voronoi region of $a_1$ contains only the elements $1$ and $2$, and the remaining elements are contained in the Voronoi region of $a_2$ implying
\[a_1=a[1,2]=\frac 43, \te{ and } a_2=a[3,6]=\frac{31}{8}\]
with quantization error $V_2=\frac{341}{768}$. Thus, the proof of the proposition is complete.
\end{proof}
\begin{prop} \label{prop13}
The optimal set of three-means is given by $\set{1, a[2,3], a[4,6]}$ with quantization error $V_3=\frac{65}{384}$.
\end{prop}
\begin{proof}
Notice that $a[2,3]=\frac 73$, and $a[4,6]=\frac{19} 4$.
The distortion error due to the set $\gb:=\set{1, \frac 73, \frac {19} 4}$ is given by
\begin{align*}\sum_{j=1}^6 f(j)\min_{a\in \gb} (j-a)^2=\sum _{j=2}^3 f(j)\Big(j-\frac{7}{3}\Big)^2+\sum _{j=4}^6f(j)\Big(j-\frac{19}{4}\Big)^2=\frac{65}{384}.
\end{align*}
Since $V_3$ is the quantization error for three-means, we have $V_3\leq \frac{65}{384}=0.169271$. Let $\ga:=\set{a_1, a_2, a_3}$ be an optimal set of three-means such that $1\leq a_1<a_2<a_3\leq 6$. Notice that the Voronoi region of $a_1$ must contain $1$. Suppose that the Voronoi region of $a_1$ also contains $3$. Then,
\[V_3\geq \sum_{j=1}^3 f(j)(j-a[1,3])^2 =\frac{13}{28}>V_3,\]
which yields a contradiction. Thus, we can assume that the Voronoi region of $a_1$ does not contain $3$. Suppose that the Voronoi region of $a_1$ contains only the two elements $1$ and $2$. Then, the Voronoi region of $a_2$ must contain $3$. The following two cases can arise:

\tit{Case~1. The Voronoi region of $a_2$ does not contain 4.}

Then, we must have $a_2=3$, and $a_4=a[4, 6]$ yielding
\[V_3\geq \sum_{j=1}^2 f(j) (j-a[1,2])^2 +\sum_{j=4}^6 f(j)(j-a[4,6])^2 =\frac{97}{384}=0.252604>V_3,\]
which is a contradiction.

\tit{Case~2. The Voronoi region of $a_2$ contains 4.}

Then, \[V_3\geq  \sum_{j=1}^2 f(j) (j-a[1,2])^2 +\sum_{j=3}^4 f(j)(j-a[3,4])^2 =\frac 5 {24} =0.208333>V_3,\]
which leads to a contradiction.

Hence, by Case~1 and Case~2, we can assume that the Voronoi region of $a_1$ contains only the element $1$, i.e., $a_1=1$. Then, the Voronoi region of $a_2$ must contain $2$. Suppose that the Voronoi region of $a_2$ also contains the set $\set{2,3,4}$. Then,
\[V_3\geq \sum_{j=2}^4 f(j)(j-a[2,4])^2 =\frac{13}{56}=0.232143>V_3,\]
which yields a contradiction. Thus, we can assume that the Voronoi region of $a_2$ does not contain $4$. Suppose that the Voronoi region of $a_2$ contains only the element $2$. Then, the Voronoi region of $a_3$ must contain the remaining elements, which yields
\[V_3\geq  \sum_{j=3}^6 f(j)(j-a[3,6])^2 =\frac{71}{256}=0.277344>V_3,\]
which is a contradiction. Hence, we can assume that the Voronoi region of $a_2$ contains only the two elements $2$ and $3$ implying the fact that the Voronoi region of $a_3$ contains the elements $4$, $5$, and $6$. Thus, we have
\[a_1=1, \, a_2=a[2,3]=\frac 73, \te{ and } a_3=a[4,6]=\frac {19}4,\]
with quantization error $V_3=\frac{65}{384}$, which yields the proposition.
\end{proof}

\begin{prop} \label{prop13}
The optimal set of four-means is  $\set{1, 2, a[3,4], a[5,6]}$ with quantization error $V_4=\frac{11}{192}$.
\end{prop}
\begin{proof}
The distortion error due to the set $\gb:=\set{1, 2, a[3,4], a[5,6]}$ is given by
\begin{align*} \sum_{j=1}^6 f(j)\min_{a\in \gb} (j-a)^2=\sum _{j=3}^4f(j)(j-a[3,4])^2+\sum _{j=5}^6f(j)(j-a[5,6])^2 =\frac{11}{192}.
\end{align*}
Since $V_4$ is the quantization error for four-means, we have $V_4\leq\frac{11}{192}=0.0572917$.
Let $\ga:=\set{a_1, a_2, a_3, a_4}$ be an optimal set of four-means. Without any loss of generality, we can assume that $1\leq a_1<a_2<a_3<a_4\leq 6$. The Voronoi region of $a_1$ must contain $1$. Suppose that the Voronoi region of $a_1$ contains $2$ as well. Then,
\[V_4\geq \sum _{j=1}^2f(j)(j-a[1,2])^2=\frac 1 6>V_4,\]
which gives a contradiction. Hence, we can assume that the Voronoi region of $a_1$ contains only the element $1$, i.e., $a_1=1$. Then, the Voronoi region of $a_2$ must contain
$2$. Suppose that the Voronoi region of $a_2$ also contains $3$. Then,
\[V_4\geq \sum _{j=2}^3f(j)(j-a[2,3])^2=\frac 1{12}=0.0833333>V_4,\]
which leads to a contradiction. Hence, the Voronoi region of $a_2$ does not contain $3$, i.e., $a_2=2$. Then, the Voronoi region of $a_3$ must contain $3$. Suppose that the Voronoi region of $a_3$ contains the set $\set{3,4,5}$. Then, we have
\[V_4\geq \sum _{j=3}^5f(j)(j-a[3,5])^2=\frac{13}{112}=0.116071>V_4,\]
which yields a contradiction. Thus, we can assume that the Voronoi region of $a_3$ does not contain $5$. Suppose that the Voronoi region of $a_3$ contains $3$ only. Then, the Voronoi region of $a_5$ contains $4, 5, 6$, which implies
\[V_4=\sum_{j=4}^6 f(j)(j-a[4, 6])^2 =\frac{11}{128}=0.0859375>V_4,\]
which gives a contradiction. Hence, the Voronoi region of $a_3$ contains $3$ and $4$ yielding $a_3=a[3, 4]$, and $a_4=a[5,6]$. Thus, the optimal set of four-means is  $\set{1, 2, a[3,4], a[5,6]}$ with quantization error $V_4=\frac{11}{192}$,  which is the proposition.
\end{proof}

Using the similar technique as the previous proposition, the following proposition can be proved.
\begin{prop} \label{prop14}
The optimal set of five-means is  $\set{1, 2,3,4, a[5,6]}$ with quantization error $V_5=\frac{1}{64}$.
\end{prop}

\subsection{Optimal quantization for the probability distribution $P$ with $x=\frac7{10}$.}
In this case, the probability mass function $f$ is given by
\[f(j)=P(X=j)=\left\{\begin{array} {cc}\vspace {0.05 in}
\frac 7{10}  \te{ if } j=1, \\  
\vspace {0.05 in}  
 (\frac 3{10})^{j-1}\frac 7{10} \te{ if } j\in \set{ 2, 3, 4, 5}, \\  \vspace {0.05 in}
(\frac 3{10})^{5} \te { if } j=6,\\  \vspace {0.05 in} 
0 \te{ otherwise.}
\end{array}
\right.
 \]
 We now give the following propositions.
\begin{prop} \label{prop22}
The optimal set of two-means is given by $\set{1, a[2, 6]}$ with quantization error $V_2=\frac{174296997}{1000000000}$.
\end{prop}

\begin{proof} The distortion error due to the set $\gb:=\set{1, a[2, 6]}$ is given by
\begin{align*} \sum_{j=1}^6 f(j)\min_{a\in \gb} (j-a)^2=\sum _{j=2}^6 f(j)(j-a[2, 6])^2=\frac{174296997}{1000000000}.
\end{align*}
Since $V_2$ is the quantization error for two-means, we have $V_2\leq \frac{174296997}{1000000000}=0.174296997$.
Let $\ga:=\set{a_1, a_2}$ be an optimal set of two-means. Without any loss of generality, we can assume that $1\leq a_1<a_2\leq 6$. Notice that the Voronoi region of $a_1$ must contain $1$. Suppose that the Voronoi region of $a_1$ contains the set $\set{1,2,3}$. Then,
\[V_2\geq \sum _{j=1}^3f(j)(j-a[1, 3])^2=\frac{4809}{13900}=0.345971>V_2,\]
which gives a contradiction. Hence, we can assume that the Voronoi region of $a_1$ does not contain $3$. Next, suppose that the Voronoi region of $a_1$ contains the set $\set{1,2}$. Then, the Voronoi region of $a_2$ contains all the remaining elements, and so
\[V_2=\sum _{j=1}^2f(j)(j-a[1,2])^2+\sum _{j=3}^6f(j)(j-a[3, 6])^2=\frac{272139987}{1300000000}=0.209338>V_2,\]
which yields a contradiction.
Hence, we can assume that the Voronoi region of $a_1$ contains only the element $1$, and the remaining elements are contained in the Voronoi region of $a_2$ implying
\[a_1=1, \te{ and } a_2=a[2,6]\]
with quantization error $V_2=\frac{174296997}{1000000000}$. Thus, the proof of the proposition is complete.
\end{proof}

\begin{prop} \label{prop23}
The optimal set of three-means is given by $\set{1, 2, a[3,6]}$ with quantization error $V_3=\frac{4779999}{100000000}$.
\end{prop}
\begin{proof}
The distortion error due to the set $\gb:=\set{1, 2, a[3,6]}$ is given by
\begin{align*} \sum_{j=3}^6 f(j)\min_{a\in \gb} (j-a)^2=\sum _{j=3}^6 f(j)(j-a[3, 6])^2=\frac{4779999}{100000000}=0.04779999.
\end{align*}
Since $V_3$ is the quantization error for three-means, we have $V_3\leq 0.04779999$. Let $\ga:=\set{a_1, a_2, a_3}$ be an optimal set of three-means such that $1\leq a_1<a_2<a_3\leq 6$. Notice that the Voronoi region of $a_1$ must contain $1$. Suppose that the Voronoi region of $a_1$ also contains $2$. Then,
\[V_3\geq \sum_{j=1}^2 f(j)(j-a[1,2])^2 =\frac{21}{130}=0.161538>V_3,\]
which yields a contradiction. Thus, we can assume that the Voronoi region of $a_1$ contains only the element $1$, i.e., $a_1=1$. The Voronoi region of $a_2$ contains $2$. Suppose that the Voronoi region of $a_2$ also contains the set $\set{2,3}$. Then,
\[V_3\geq \sum_{j=2}^3 f(j)(j-a[2,3])^2=\frac{63}{1300}=0.0484615>V_3,\]
which is a contradiction. Hence, the Voronoi region of $a_2$ contains only the element $2$, which yields $a_2=2$, and $a_3=a[3,6]$, with quantization error $V_3=\frac{4779999}{100000000}$. Thus, the proof of the proposition is complete.
\end{proof}
Following the similar techniques as given in Proposition~\ref{prop23}, we can prove the following two propositions.
 \begin{prop} \label{prop24}
The optimal set of four-means is given by $\set{1, 2, 3, a[4,6]}$ with quantization error $V_4=\frac{112833}{10000000}$.
\end{prop}
 \begin{prop} \label{prop25}
The optimal set of five-means is given by $\set{1, 2, 3, 4, a[5,6]}$ with quantization error $V_5=\frac{1701}{1000000}$.
\end{prop}

\section{Optimal quantization for infinite discrete distributions} \label{sec2}

In this section, for $n\in\D N$, we investigate the optimal sets of $n$-means for two different infinite discrete distributions. We give them in the following two subsections.

\subsection{Optimal quantization for an infinite discrete distribution with support $\set{\frac 1 n : n\in \D N}$} \label{sec3}

Let $\D N:=\set{1, 2, 3, \cdots}$ be the set of natural numbers, and let $P$ be a Borel probability measure on the set $\set{\frac 1 n : n \in \D N}$ with probability mass function $f$ given by
\[f(x)=\left\{\begin{array} {ll} \vspace{0.05 in}
\frac 1{2^k} & \te{ if } x =\frac 1 k \te { for } k\in \D N,\\ \vspace{0.05 in}
0 &  \te{ otherwise}.
\end{array}
\right.\]
Then, $P$ is a Borel probability measure on $\D R$, and the support of $P$ is given by supp$(P)=\set{\frac 1 n :  n \in \D N}$.  In this section, for the probability measure $P$, we investigate the optimal sets of $n$-means and the $n$th quantization errors for $n\in \D N$. For $k, \ell\in \D N$, where $k\leq \ell$, write
\[[k, \ell]:=\set{\frac 1 n : n \in \D N \te{ and } k\leq n\leq \ell}, \te{ and } [k, \infty):=\set{\frac 1 n : n\in \D N \te{ and } n\geq k}.\]
Further, write
\[Av[k, \ell]: =E\Big (X : X \in [k, \ell]\Big)=\frac{\sum _{n=k}^{\ell} \frac{1}{2^n} \frac 1 n}{\sum_{n=k}^{\ell}\frac{1}{2^n}}, \  Av[k, \infty): =E\Big (X : X \in [k, \infty)\Big)=\frac{\sum _{n=k}^{\infty} \frac{1}{2^n} \frac 1 n}{\sum_{n=k}^{\infty}\frac{1}{2^n}},\]
\[Er[k, \ell]:=\sum _{n=k}^{\ell} \frac{1}{2^n} \Big(\frac 1 n-Av[k, \ell]\Big)^2, \te{ and } Er[k, \infty):=\sum _{n=k}^{\infty} \frac{1}{2^n} \Big(\frac 1 n-Av[k,\infty)\Big)^2.\]
 Notice that
$E(X):=E(X : X \in \te{supp}(P)) =\sum _{n=1}^\infty \frac 1{2^n} \frac 1 n=Av[1, \infty)=\log(2)$, and so the optimal set of one-mean is the set $\set{\log(2)}$ with quantization error
\[V(P)=\sum _{n=1}^{\infty} \frac 1 {2^n} \Big (\frac{1}{n}-\log(2)\Big)^2=Er[1, \infty) =\frac{1}{12} \left(\pi ^2-18 \log ^2(2)\right)=0.101788.\]

\begin{prop} \label{prop01}
The set $\set{Av[2, \infty), 1}$ forms the optimal set of two-means for the probability measure $P$ with quantization error $V_2(P)=Er[2, \infty)=\frac{1}{12} \left(\pi ^2-12-30 \log ^2(2)+24 \log (2)\right)=0.0076288597$.
\end{prop}
\begin{proof}
Consider the set $\gb:=\set{Av[2, \infty), 1}$. Since $\frac 13<\frac 12 (Av[2, \infty)+1)<1$, the Voronoi region of $1$ contains only the element $1$, and the Voronoi region of $Av[2, \infty)$ contains the set $\set{\frac 1 n : n\geq 2}$. Hence, the distortion error due to the set $\gb$ is given by
\[V(P; \gb)=Er[2, \infty)=\frac{1}{12} \left(\pi ^2-12-30 \log ^2(2)+24 \log (2)\right)=0.0076288597.\]
Since $V_2(P)$ is the quantization error for two-means, we have $V_2(P)\leq 0.0076288597$. Let $\ga:=\set{a_2, a_1}$ be an optimal set of two-means. Due to Proposition~\ref{prop0}, we can assume that $0\leq a_2<a_1\leq 1$. The Voronoi region of $a_1$ must contain $1$. Suppose that the Voronoi region of $a_1$ also contains $\frac 12$. Then,
\[V_2(P)\geq Er[1,2]=\frac{1}{24}=0.0416667>V_2(P),\]
which leads to a contradiction. Hence, we can assume that the Voronoi region of $a_1$ does not contain $\frac 12$. Again, by Proposition~\ref{prop0}, the Voronoi region of $a_2$ cannot contain the element $1$. Thus, we have
$a_2=Av[2, \infty)$, and $a_1=1$, and the corresponding quantization error is $V_2(P)=Er[2, \infty)=0.0076288597$. Thus, the proof of the proposition is complete.
\end{proof}

\begin{prop} \label{prop02}
The set $\set{Av[3, \infty), \frac 12, 1}$ forms the optimal set of three-means for the probability measure $P$ with quantization error $V_3(P)=Er[3, \infty)=0.00116437359$.
\end{prop}
\begin{proof}
Consider the set $\gb:=\set{Av[3, \infty), \frac 12, 1}$. Since, $\frac 13 <\frac 12 (Av[3, \infty)+\frac 12)<\frac 12$, and $\frac 12<\frac 12(\frac 12+1)<1$, the distortion error due to the set $\gb:=\set{Av[3, \infty), \frac 12, 1}$ is given by
\[V(P; \gb)=Er[3, \infty)=\frac{1}{24} \left(2 \pi ^2-51-108 \log ^2(2)+120 \log (2)\right)=0.00116437359.\]
Since $V_3(P)$ is the quantization error for three-means, we have $V_3(P)\leq 0.00116437359$. Let $\ga:=\set{a_3, a_2, a_1}$ be an optimal set of three-means such that $0\leq a_3<a_2<a_1\leq 1$. Proceeding as Proposition~\ref{prop01}, we can show that $a_1=1$. Suppose that the Voronoi region of $a_2$  contains $\frac 12$ and $\frac 13$. Then,
  \[V_3(P)\geq Er[2,3]=0.002314814815>V_3(P),\]
which is a contradiction. Hence, the Voronoi region of $a_2$ cannot contain $\frac 13$. Thus, we have
\[a_3=Av[3, \infty), \, a_2=\frac 12, \te{ and } a_1=1,\]
with quantization error $V_3(P)=Er[3, \infty)=0.00116437359$.
Thus, the proof of the proposition is complete.
\end{proof}

\begin{prop} \label{prop03}
The set $\set{Av[4, \infty), \frac 13, \frac 12, 1}$ forms the optimal set of four-means for the probability measure $P$ with quantization error $V_4(P)=Er[4, \infty)=0.0002418966477$.
\end{prop}

\begin{proof}
The proof of this proposition is similar to the proof of Proposition~\ref{prop02}.
\end{proof}

\begin{prop} \label{prop04}
The set $\set{Av[5, \infty), \frac 14, \frac 13, \frac 12, 1}$ forms the optimal set of five-means for the probability measure $P$ with quantization error $V_5(P)=Er[5, \infty)=0.00005991266593$.
\end{prop}
\begin{proof}
The distortion error due to the set $\gb:=\set{Av[5, \infty), \frac 14, \frac 13, \frac 12, 1}$ is given by
\[V(P; \gb):=Er[5, \infty)=Er[5, \infty)=0.00005991266593.\]
Since $V_5(P)$ is the quantization error for five-means, we have $V_5(P)\leq 0.00005991266593$. Let $\ga:=\set{a_5, a_4, a_3, a_2, a_1}$ be an optimal set of five-means such that $0 \leq a_5<a_4<a_3<a_2<a_1\leq 1$. Proceeding as Proposition~\ref{prop02}, we can show that $a_1=1, \, a_2=\frac 12$, and $a_3=\frac 13$. We now show that $a_4=\frac 14$.
Suppose that the Voronoi region of $a_4$ contains $\frac 14, \, \frac 15$, and $\frac 16$.
Then,
\[V_5(P)\geq Er[4, 6]=0.0001116071429>V_5(P),\]
which is a contradiction. Assume that the Voronoi region of $a_4$ contains only the elements $\frac 14$, and $\frac 15$. Then, the Voronoi region of $a_5$ contains the set $[6, \infty)$, and so we have
 \[V_5(P)=Er[6, \infty)+Er[4, 5]=0.00006872664638>V_5(P),\]
 which leads to a contradiction. Hence, we can assume that the Voronoi region of $a_4$ contains only the element $\frac 14$. Thus, we have
$a_5=Av[5, \infty), \, a_4=\frac 14, \, a_3=\frac 13, \, a_2=\frac 12, \te{ and } a_1=1$  with quantization error $V_5(P)=Er[5, \infty)=0.00005991266593$. Thus, the proof of the proposition is complete.
\end{proof}

\begin{prop} \label{prop05}
The set $\set{Av[7, \infty), Av[5, 6],  \frac 14, \frac 13, \frac 12, 1}$ forms the optimal set of six-means for the probability measure $P$ with quantization error \[V_6(P)=Er[7, \infty)+Er[5, 6]=0.00001658886625.\]
\end{prop}

\begin{proof}
Notice that $\frac 1 7=0.142857<\frac 12(Av[7, \infty)+ Av[5, 6])=0.158488<0.166667=\frac 16$, and $\frac{1}{5}<\frac{1}{2} (Av[5, 6]+\frac{1}{4})<\frac{1}{4}$. Hence, the distortion error due to the set $\gb:=\set{Av[7, \infty), Av[5,6], \frac 14, \frac 13, \frac 12, 1}$ is given by
\[V(P; \gb)=Er[7, \infty)+Er[5, 6]=0.00001658886625.\]
Since $V_6(P)$ is the distortion error for six-means, we have $V_6(P)\leq 0.00001658886625$. Let $\ga:=\set{a_6, a_5, a_4, a_3, a_2, a_1}$ be an optimal set of six-means such that $0\leq a_6<a_5<\cdots<a_1\leq 1$. Proceeding in the similar way as in the proof of Proposition~\ref{prop02}, we can show that $a_3=\frac 13, \, a_2=\frac 12$, and $a_1=1$. Proceeding in the similar way as in the proof of Proposition~\ref{prop04}, we can show that $a_4=\frac 1 4$. We now show that $a_5=Av[5, 6]$. Notice that the Voronoi region of $a_5$ must contain $\frac 1 5$.
Suppose that the Voronoi region of $a_5$ contains $\frac 17$ and $\frac 16$ as well. Then,
\[V_6(P)\geq Er[5, 7]=0.00002576328150>V_6(P),\]
which leads to a contradiction. Suppose that the Voronoi region of $a_5$ contains only the element $\frac 15$, i.e., $a_5=\frac 15$. Then,
\[V_6(P)=Er[6, \infty)=0.00001664331305>V_6(P),\]
which yields a contradiction. Hence, we can assume that the Voronoi region of $a_5$ contains only the two elements $\frac 16$ and $\frac 15$. Thus, we have
\[a_6=Av[7, \infty), \, a_5=Av[5, 6], \, a_4=\frac 14, \, a_3=\frac 13, \, a_2=\frac 12, \te{ and } a_1=1,\]
and the quantization error is $V_6(P)=Er[7, \infty)+Er[5, 6]=0.00001658886625$. Thus, the proof of the proposition is complete.
\end{proof}

In the following proposition, we calculate the optimal set of $n$-means and the $n$th quantization error for $n=300$.
\begin{prop} \label{prop06}
The set $\set{Av[301, \infty), Av[299, 300], \frac 1 {298}, \frac 1{297},  \cdots, \frac 13, \frac 12, 1}$ forms the optimal set of $300$-means for the probability measure $P$ with quantization error $V_{300}(P)=Er[301, \infty)+Er[299, 300]=1.564317642582409606174128\times 10^{-100}$.
\end{prop}
 
\begin{proof}
Notice that \[\frac 1 {301}=0.003322259136<\frac 12(Av[301, \infty)+ Av[299, 300])=0.003326047849<0.003333333333=\frac 1{300},\] and $\frac{1}{299}=0.003344481605<\frac{1}{2} (Av[299, 300]+\frac{1}{298})=0.003348235106<0.003355704698=\frac{1}{298}$. Hence, the distortion error due to the set
\[\gb:=\set{Av[301, \infty), Av[299, 300], \frac 1 {198}, \frac 1{197},  \cdots, \frac 13, \frac 12, 1}\] is given by
\[V(P; \gb)=Er[301, \infty)+Er[299, 300]=1.564317642582409606174128\times 10^{-100}.\]
Since $V_{300}(P)$ is the distortion error for $300$-means, we have \[V_{300}(P)\leq 1.564317642582409606174128\times 10^{-100}.\] Let $\ga:=\set{a_{300}, a_{299}, \cdots, a_3, a_2, a_1}$ be an optimal set of $300$-means such that $0\leq a_{300}<a_{299}<\cdots<a_1\leq 1$. Proceeding in the similar way as in the proof of Proposition~\ref{prop02}, we can show that $a_{297}=\frac 1{297}, \, a_{296}=\frac 1{296}, \, \cdots, \, a_3=\frac 13, \, a_2=\frac 12$, and $a_1=1$. Proceeding in the similar way as in the proof of Proposition~\ref{prop04}, we can show that $a_{298}=\frac 1{298}$.
We now show that $a_{299}=Av[299, 300]$. The Voronoi region of $a_{299}$ must contain $\frac 1 {299}$.
Suppose that the Voronoi region of $a_{299}$ contains $\frac 1 i$ for $i=299, 300, 301, 302$. Then,
\[V_{300}(P)\geq Er[299, 302]=1.953916208081117722202350\times 10^{-100}>V_{300}(P),\]
which leads to a contradiction. Assume that the Voronoi region of $a_{299}$ contains only the elements $\frac 1 i$ for $i=299, 300, 301$. Then,
\[V_{300}(P)= Er[302, \infty)+Er[299, 301]=1.698521119259119376459397\times 10^{-100}>V_{300}(P),\]
which yields a contradiction. Assume that the Voronoi region of $a_{299}$ contains only the element $\frac 1{299}$. Then,
\[V_{300}(P)=Er[300, \infty)=2.345910694878821203973953\times 10^{-100}>V_{300}(P),\]
which gives a contradiction.
Hence, we can assume that the Voronoi region of $a_{299}$ contains only the two elements $\frac 1{299}$ and $\frac 1{300}$.
Thus, we have
\[a_{300}=Av[301, \infty), \, a_{299}= Av[299, 300], \, a_{298}=\frac 1 {298}, \,  \cdots, \,  a_4=\frac 14, \, a_3=\frac 13, \, a_2=\frac 12, \te{ and } a_1=1,\]
and the quantization error is given by
\[V_{300}(P)=Er[301, \infty)+Er[299, 300]=1.564317642582409606174128\times 10^{-100}.\] Thus, the proof of the proposition is complete.
\end{proof}

We now give the following theorem.
\begin{theorem} \label{theo1}
For any positive integer $n$, the sets $\set{Av[n, \infty), \frac 1 {n-1},  \cdots, \frac 13, \frac 12, 1}$, where $1\leq n\leq 5$, form the optimal sets of $n$-means for the probability measure $P$ with quantization errors
$V_n(P):=Er[n, \infty).$
For the positive integers $n$, where $6\leq n\leq 300$, the sets $\set{Av[n+1, \infty), Av[n-1, n], \frac 1 {n-2}, \cdots, \frac 13, \frac 12, 1}$ form the optimal sets of $n$-means for the probability measure $P$ with quantization errors
\[V_n(P)=Er[n+1, \infty)+Er[n-1, n].\]
\end{theorem}

\begin{proof}

As a consequence of Proposition~\ref{prop01} through Proposition~\ref{prop04}, it follows that for \( 1 \leq n \leq 5 \), the sets
\[
\left\{ \text{Av}[n, \infty), \frac{1}{n-1}, \cdots, \frac{1}{3}, \frac{1}{2}, 1 \right\}
\]
form the optimal sets of \( n \)-means for the probability measure \( P \) with corresponding quantization errors
\[
V_n(P) = Er[n, \infty).
\]
Proceeding in the similar way as Proposition~\ref{prop05} and Proposition~\ref{prop06}, we can show that for any positive integer $n$, where $6\leq n\leq 300$, the sets $\set{Av[n+1, \infty), Av[n-1, n], \frac 1 {n-2}, \cdots, \frac 13, \frac 12, 1}$ form the optimal sets of $n$-means for the probability measure $P$ with quantization errors
\[V_n(P)=Er[n+1, \infty)+Er[n-1, n].\]
Thus, we complete the proof of the theorem.
\end{proof}

The following problem remains open. 
\begin{open}  
Proceeding in the similar way as Proposition~\ref{prop05} and Proposition~\ref{prop06}, it can be shown that the set $\set{Av[n+1, \infty), Av[n-1, n], \frac 1 {n-2}, \cdots, \frac 13, \frac 12, 1}$ also forms an optimal set of $n$-means for $n=301$. It is still not known whether the sets $\set{Av[n+1, \infty), Av[n-1, n], \frac 1 {n-2}, \cdots, \frac 13, \frac 12, 1}$ give the optimal sets of $n$-means for all positive integers $n\geq 6$. If not, then the least upper bound of $n\in \D N$ for which such sets give the optimal sets of $n$-means for the probability measure $P$ is not known yet.
\end{open}

\subsection{Optimal quantization for an infinite discrete distribution with support $\set{n : n\in \D N}$} \label{sec4}

Let $\D N:=\set{1, 2, 3, \cdots}$ be the set of natural numbers, and let $P$ be a Borel probability measure on the set $\set{n : n \in \D N}$ with probability density function $f$ given by
\[f(x)=\left\{\begin{array} {ll} \vspace {0.05 in} 
\frac 1{2^n} & \te{ if } x =n \te { for } n\in \D N,\\ \vspace {0.05 in}
0 &  \te{ otherwise}.
\end{array}
\right.\]
Then, $P$ is a Borel probability measure on $\D R$, and the support of $P$ is the set $\D N$ of natural numbers.  In this section, our goal is to determine the optimal sets of $n$-means and the $n$th quantization errors for all positive integers $n$ for the probability measure $P$. For $k, \ell\in \D N$, where $k\leq \ell$, write
\[[k, \ell]:=\set{n : n \in \D N \te{ and } k\leq n\leq \ell}, \te{ and } [k, \infty):=\set{n : n\in \D N \te{ and } n\geq k}.\]
Further, write
\[Av[k, \ell]: =E\Big (X : X \in [k, \ell]\Big)=\frac{\sum _{n=k}^{\ell} \frac{n}{2^n}}{\sum_{n=k}^{\ell}\frac{1}{2^n}}, \  Av[k, \infty): =E\Big (X : X \in [k, \infty)\Big)=\frac{\sum _{n=k}^{\infty} \frac{n}{2^n}}{\sum_{n=k}^{\infty}\frac{1}{2^n}},\]
\[Er[k, \ell]:=\sum _{n=k}^{\ell} \frac{1}{2^n} \Big(n-Av[k, \ell]\Big)^2, \te{ and } Er[k, \infty):=\sum _{n=k}^{\infty} \frac{1}{2^n} \Big(n-Av[k,\infty)\Big)^2.\]
 Notice that
$E(P):=E(X : X \in \te{supp}(P)) =\sum _{n=1}^\infty \frac n{2^n}  =Av[1, \infty)=2$, and so the optimal set of one-mean is the set $\set{2}$ with quantization error
\[V(P)=\sum _{n=1}^{\infty} \frac 1 {2^n} (n-2)^2=Er[1, \infty) =2.\]

\begin{prop} \label{prop43}
The optimal set of two-means is given by $\set{Av[1,2], Av[3, \infty)}$ with quantization error $V_2=\frac{2}{3}$.
\end{prop}
\begin{proof} We see that $Av[1,2]=\frac 43$, and $Av[3, \infty)=4$.
Since $\frac 43<\frac 12(\frac 43+4)<4$, the distortion error due to the set $\gb:=\set{\frac 43, 4}$ is given by
\[V(P; \gb)=Er[1,2]+Er[3, \infty)=\frac 23.\]
Since $V_2$ is the quantization error for two-means, we have $V_2\leq \frac 23$. Let $\ga:=\set{a_1, a_2}$, where $1\leq a_1<a_2<\infty$, be an optimal set of two-means. Notice that the Voronoi region of $a_1$ must contain $1$. Suppose that the Voronoi region of $a_1$ contains the set $\set{1, 2, 3, 4}$. Then,
\[V_2\geq \sum_{j=1}^4\frac 1 {2^j} (j-Av[1,4])^2=Er[1,4]=\frac{97}{120}=0.808333>V_2,\]
which yields a contradiction. Hence, we can assume that the Voronoi region of $a_1$ contains only the set $\set{1,2,3}$, and so the Voronoi region of $a_2$ contains the set $\set{n : n\geq 4}$. Then, we have
\[V_2=Er[1,3]+Er[4, \infty)=\frac{5}{7}=0.714286>V_2,\]
which is a contradiction. Next, suppose that the Voronoi region of $a_1$ contains only the element $1$, and so the Voronoi region of $a_2$ contains the set $\set{n : n\geq 2}$. Then, we have
\[V_2=Er[2, \infty)=1>V_2,\]
which leads to a contradiction. Hence, we can assume that the Voronoi region of $a_1$ contains the set $\set{1, 2}$, and so the Voronoi region of $a_2$ contains $\set{3, 4, 5, \cdots}$ yielding $a_1=Av[1,2]$, and $a_2=Av[3, \infty)$, and the corresponding quantization error is $V_2=\frac 23$. Thus, the proof of the proposition is complete.
\end{proof}

\begin{prop} \label{prop111}
The sets $\set{1, Av[2,3], Av[4, \infty)}$, and $\set{Av[1,2], Av[3,4], Av[5, \infty)}$  form two different optimal sets of three-means with quantization error $V_3=\frac{1}{3}$.
\end{prop}
\begin{proof}
The distortion error due to set $\gb:=\set{1, Av[2,3], Av[4, \infty)}$ is given by
\[V(P; \gb)=Er[2,3]+Er[4, \infty)=\frac 13.\]
Notice that the distortion error due to the set $\set{Av[1,2], Av[3,4], Av[5, \infty)}$ is also $\frac 13$.
Since $V_3$ is the quantization error for three-means, we have $V_3\leq \frac 13$. Let $\ga:=\set{a_1, a_2, a_3}$ be an optimal set of three-means, where $1\leq a_1<a_2<a_3<\infty$. Suppose that the Voronoi region of $a_1$ contains the set $\set{1, 2, 3}$.
Then,
\[V_3\geq \sum_{j=1}^3 \frac 1 {2^j}(j-Av[1,3])^2 =\frac{13}{28}>\frac 13>V_3,\]
which leads to a contradiction. Hence, we can assume that the Voronoi region of $a_1$ contains either the set $\set{1}$, or the set $\set{1,2}$. Consider the following two cases:

\tit{Case 1. The Voronoi region of $a_1$ contains only the set $\set{1}$.}

In this case, the Voronoi region of $a_2$ must contain the element $2$. Suppose that the Voronoi region of $a_2$ contains the set $\set{2, 3, 4, 5}$. Then,
\[V_3\geq  \sum_{j=2}^5\frac 1 {2^j}(j-Av[2,5])^2 =\frac{97}{240}=0.404167>V_3,\]
which yields a contradiction. Assume that the Voronoi region of $a_2$ contains only the set $\set{2,3,4}$, and so the Voronoi region of $a_3$ contains the set $\set{n : n\geq 5}$. Then, the distortion error is
\[V_3=Er[2,4]+Er[5, \infty)=\frac 5 {14} =0.357143>V_3,\]
which gives a contradiction.
Next, assume that the Voronoi region of $a_2$ contains only the element $2$, and so the Voronoi region of $a_3$ contains the set $\set{n : n\geq3}$. Then, the distortion error is
\[V_3=Er[3, \infty) =\frac{1}{2}>V_3,\]
which is a contradiction.
Hence, in this case, we can conclude that the Voronoi region of $a_2$ contains only the set $\set{2,3}$ yielding $a_1=1$, $a_2=Av[2,3]$, and $a_3=Av[4, \infty)$ with quantization error $V_3=\frac 13$.

\tit{Case 2. The Voronoi region of $a_1$ contains only the set $\set{1,2}$.}

In this case, the Voronoi region of $a_2$ must contain the element $3$. Suppose that the Voronoi region of $a_2$ contains the set $\set{3, 4, 5, 6}$. Then,
\[V_3\geq  \sum_{j=1}^2\frac 1 {2^j}(j-Av[1,2])^2+ \sum_{j=3}^6\frac 1 {2^j}(j-Av[3, 6])^2 =\frac{59}{160} =0.36875>V_3,\]
which yields a contradiction. Assume that the Voronoi region of $a_2$ contains only the set $\set{3, 4, 5}$, and so the Voronoi region of $a_3$ contains the set $\set{n : n\geq 6}$. Then, the distortion error is
\[V_3=Er[1,2]+Er[3,5]+Er[6, \infty)=\frac{29}{84} =0.345238>V_3,\]
which gives a contradiction.
Next, assume that the Voronoi region of $a_2$ contains only the element $3$, and so the Voronoi region of $a_3$ contains the set $\set{n : n\geq 4}$. Then, the distortion error is
\[V_3=Er[1,2]+Er[4, \infty)=\frac{5}{12}=0.416667>V_3,\]
which yields a contradiction.
Hence, in this case, we can conclude that the Voronoi region of $a_2$ contains only the set $\set{3,4}$ yielding $a_1=Av[1,2]$, $a_2=Av[3,4]$, and $a_3=Av[5, \infty)$ with quantization error $V_3=\frac 13$.

By Case~1 and Case~2, the proof of the proposition is complete.
\end{proof}

We need the following lemma.

\begin{lemma} \label{lemma1}
Let $n\geq 4$, and let $\ga_n$ be an optimal set of $n$-means. Then, $\ga_n$ must contain the set $\set{1, 2, \cdots, (n-3)}$.
\end{lemma}
\begin{proof}
The distortion error due to the set $\gb:=\set{1, 2, \cdots, (n-3), (n-2), Av[n-1, n], Av[n+1, \infty)}$ is given by
\[V(P; \gb)=Er[n-1, n]+Er[n+1, \infty)=\frac{2^{3-n}}{3}.\]
Since $V_n$ is the quantization error for $n$-means, we have $V_n\leq \frac{2^{3-n}}{3}$.
Let $\ga_n:=\set{a_1, a_2, \cdots, a_n}$ be an optimal set of $n$-means such that $1\leq a_1<a_2<\cdots<a_n<\infty$. We show that $a_1=1, a_2=2, \cdots, a_{n-3}=n-3$.
Notice that the Voronoi region of $a_1$ must contain the element $1$. Suppose that the Voronoi region of $a_1$ also contains the element $2$. Then,
\[V_n>\sum_{j=1}^2\frac 1 {2^j}(j-Av[1,2])^2=\frac{1}{6}\geq \frac{2^{3-n}}{3} \geq V_n,\]
which is a contradiction. Hence, we can conclude that the Voronoi region of $a_1$ contains only the element $1$ yielding $a_1=1$. Thus, we can deduce that there exists a positive integer $k$, where $1\leq k<n-3$, such that $a_1=1, a_2=2, \cdots, a_k=k$.
We now show that $a_{k+1}=k+1$. Notice that the Voronoi region of $a_{k+1}$ must contain $k+1$. Suppose that the Voronoi region of $a_{k+1}$ also contains the element $k+2$. Then, as $k<n-3$, we have
\[V_n>\sum_{j=k+1}^{k+2}\frac 1 {2^j}(j-Av[k+1, k+2])^2=\frac{2^{-k-1}}{3}\geq \frac{2^{3-n}}{3}\geq V_n,\]
which is a contradiction. Hence, we can conclude that the Voronoi region of $a_{k+1}$ contains only the element $k+1$ yielding $a_{k+1}=k+1$. Thus, by the Principle of Mathematical Induction, we deduce that  $a_1=1, a_2=2, \cdots, a_{n-3}=n-3$. Thus, the proof of the lemma is complete.
\end{proof}

\begin{theorem}
Let $n\geq 4$, and let $\ga_n$ be an optimal set of $n$-means. Then, either $\ga_n=\set{1,2,3, \cdots, n-3, n-2, Av[n-1, n], Av[n+1, \infty)}$, or $\ga_n=\set{1,2,3, \cdots, n-3, Av[n-2, n-1], Av[n, n+1], Av[n+2, \infty)}$ with quantization error $V_n=\frac{2^{3-n}}{3}$.
\end{theorem}
\begin{proof}
As shown in the proof of Lemma~\ref{lemma1}, we have $V_n\leq \frac{2^{3-n}}{3}$. Let $\ga_n:=\set{a_1, a_2, \cdots, a_n}$ be an optimal set of $n$-means such that $1\leq a_1<a_2<\cdots<a_n<\infty$. By Lemma~\ref{lemma1}, we have $a_1=1, a_2=2, \cdots, a_{n-3}=n-3$. Recall that $n\geq 4$. Suppose that the Voronoi region of $a_{n-2}$ contains the set $\set{n-2, n-1, n}$.
Then,
\[V_n\geq \sum_{j=n-2}^{n} \frac 1 {2^j}(j-Av[n-2, n])^2 =\frac{13}{7} 2^{1-n}>\frac{2^{3-n}}{3}\geq V_n,\]
which leads to a contradiction. Hence, we can assume that the Voronoi region of $a_{n-2}$ contains either the set $\set{n-2}$, or the set $\set{n-2, n-1}$. Consider the following two cases:

\tit{Case 1. The Voronoi region of $a_{n-2}$ contains only the set $\set{n-2}$.}

 Proceeding along the similar lines as Case~1 in the proof of Proposition~\ref{prop111}, we can show that the Voronoi region of $a_{n-1}$ contains only the set $\set{n-1, n}$ yielding $a_{n-2}=n-2$, $a_{n-1}=Av[n-1, n]$, and $a_n=Av[n+1, \infty)$ with quantization error $V_n=\frac{2^{3-n}}{3}$.

\tit{Case 2. The Voronoi region of $a_{n-2}$ contains only the set $\set{n-2, n-1}$.}

 Proceeding along the similar lines as Case~2 in the proof of Proposition~\ref{prop111}, we can show that the Voronoi region of $a_{n-1}$ contains only the set $\set{n, n+1}$ yielding $a_{n-2}=Av[n-2, n-1]$, $a_{n-1}=Av[n, n+1]$, and $a_n=Av[n+2, \infty)$ with quantization error $V_n=\frac{2^{3-n}}{3}$.

By Case~1 and Case~2, the proof of the theorem is complete.
\end{proof}

\section{Probability distributions when the optimal sets are given} \label{sec3}

Let $P$ be a discrete probability measure on $\D R$ with support a finite or an infinite set $\set{1, 2, 3,\cdots}$. Let $(p_1, p_2, p_3,\cdots)$ be a probability vector associated with $\set{1, 2, 3, \cdots}$ such that the probability mass function $f$ of $P$ is given by $f(k)=p_k$ if $k\in \set{1, 2, 3, \cdots}$, and zero otherwise. For $k, \ell\in \set{1, 2, 3, \cdots}$ with $k\leq \ell$, write
\[[k, \ell]:=\set{n : k\leq n\leq \ell},\te{ and } [k, \infty):=\set{k, k+1, \cdots}.\]
For a random variable $X$ with distribution $P$, let $Av[k, \ell]$ represent the conditional expectation of $X$ given that $X$ takes values on the set $\set{k, k+1, k+2, \cdots, \ell}$, i.e.,
\[Av[k, \ell]=E(X : X\in [k, \ell]),\]
where $k, \ell\in \set{1, 2, 3, \cdots}$ with $k\leq \ell$. On the other hand, by $Av[k, \infty)$ it is meant $Av[k, \infty)=E(X : X\in [k, \infty))$, where $k\in \set{1, 2, 3, \cdots}$.
Let $\ga_n$ be an optimal set of $n$-means for $P$, where $n\in \D N$. In this section, our goal is to find a set of probability vectors $(p_1, p_2, p_3,\cdots)$ such that for all $n\in \D N$, the optimal sets of $n$-means are given by
$\ga_n=\set{1, 2, 3, \cdots, n-1, Av[n, \infty)}.$

Consider the following two cases:

\tit{Case 1. $\set{1, 2, 3,\cdots}$ is a finite set.}

In this case, there exists a positive integer $m$, such that the support of $P$ is given by $\set{1, 2, 3, \cdots, m}$. Notice that for any $k\in \set{1, 2, \cdots, m}$, in this case by $[k, \infty)$ it meant the set $[k, m]$. If $m=1$, then $\ga_1=\set{1}$; and if $m=2$, then $\ga_1=\set{Av[1, \infty)}$, and $\ga_2=\set{1, Av[2, \infty)}$, i.e., there is nothing to prove as in the cases of $m=1$ and $m=2$, they are true for any associated probability vector. So, we can assume that $m\geq 3$. Define the probability vector $(p_1, p_2, \cdots, p_m)$ as follows:
\begin{align}
\label{eq340}
p_j=\left\{\begin{array}{ccc} \vspace{0.05 in}
x &\te{ if } j=1,  \\\vspace{0.05 in}
(1-x)^{j-1} x & \te{ if } 2\leq j\leq m-1,\\\vspace{0.05 in}
(1-x)^{j-1} & \te{ if } j=m.
\end{array}
\right.
\end{align}
For the sets $\ga_n$ to form the optimal sets of $n$-means for all $1\leq n\leq m$, we must have
\begin{equation} \label{eq345} (n-1)\leq \frac 12 (n-1+Av[n, \infty))\leq n
\end{equation}
and 
\begin{equation} \label{eq346} 
V(P;  \set{n-1, Av[n,\infty)}) \leq V(P; \set{Av[n-1,n], Av[n+1, \infty)})
\end{equation}
for $2\leq n\leq m$. In this regard we give two examples:  Example~\ref{ex1} and Example~\ref{ex2}.

\tit{Case 2. $\set{1, 2, 3,\cdots}$ is an infinite set.}

Define the probability vector $(p_1, p_2, p_3 \cdots)$ as follows:
\begin{align}
\label{eq440}
p_j=\left\{\begin{array}{ccc} \vspace{0.05 in}
x &\te{ if } j=1,  \\
(1-x)^{j-1} x & \te{ if } 2\leq j.
\end{array}
\right.
\end{align}
For the sets $\ga_n$ to form the optimal sets of $n$-means for all $1\leq n$, we must have
\begin{equation} \label{eq445} (n-1)\leq \frac 12 (n-1+Av[n, \infty))\leq n
\end{equation}
and 
\begin{equation} \label{eq446} 
V(P;  \set{n-1, Av[n,\infty)}) \leq V(P; \set{Av[n-1,n], Av[n+1, \infty)})
\end{equation}
for $2\leq n$. 
After some calculation, we see that there exists a real number $y$ for which the inequalities given by \eqref{eq445} and \eqref{eq446} are satisfied if $y\leq x<1$. The ten-digit rational approximation of such a $y$ is  $0.6666666667$. 
Hence, a set of probability vectors $(p_1, p_2, p_3, \cdots)$ for which the given sets $\ga_n$ form the optimal sets of $n$ means for $1\leq n $ is given by
\[\left \{\left(x, (1-x)x, (1-x)^2x, (1-x)^3x, (1-x)^4x, \cdots \right) : 0.6666666667\leq  x < 1\right\}.\]

\begin{exam} \label{ex1}
 Let $m=6$ in Case~1. Then, for $0<x<1$ we have
\begin{equation*} \label{eq31} p_1=x, \, p_2=(1-x) x, \, p_3=(1-x)^2 x, \, p_4=(1-x)^3 x, \, p_5=(1-x)^4 x, \te{ and } p_6=(1-x)^5.
\end{equation*}
After solving the inequalities given by \eqref{eq345}, we have $0.4812099363< x < 1$. Again, solving the inequality \eqref{eq346}, we have $0.6628057756\leq x<1$. Notice that $0.4812099363$ and $0.6628057756$ are the ten-digit rational approximations of two real numbers. Thus, the inequalities given by \eqref{eq345} and \eqref{eq346} are true if $0.6628057756\leq  x < 1$. Hence, a set of probability vectors $(p_1, p_2, \cdots, p_6)$ for which the given sets $\ga_n$ form the optimal sets of $n$ means for $1\leq n\leq 6$ is given by
\[\left \{\left(x, (1-x)x, (1-x)^2x, (1-x)^3x, (1-x)^4x, (1-x)^5\right) : 0.6628057756\leq  x < 1\right\}.\]
\end{exam}
\begin{exam} \label{ex2}
 Let $m=7$ in Case~1. Then, proceeding as Example~\ref{ex1},
 we see that \eqref{eq345} and \eqref{eq346} are true if $0.6654212000\leq x<1$. Hence, a set of probability vectors $(p_1, p_2, \cdots, p_7)$ for which the given sets $\ga_n$ form the optimal sets of $n$ means for $1\leq n\leq 7$ is given by
\[\left\{x, (1-x)x, (1-x)^2x, (1-x)^3x, (1-x)^4x, (1-x)^5x, (1-x)^6\right\}\]
where  $0.6654212000\leq x<1$.
\end{exam}

\section{Conclusion}
In this paper, we investigated the problem of optimal quantization for several discrete probability distributions, both finite and infinite. Starting with specific nonuniform discrete distributions supported on a finite set, we computed the optimal sets of $n$-means and the corresponding quantization errors for various values of 
$n$. We then extended our analysis to infinite discrete distributions, including those supported on the set of reciprocals of natural numbers and on the natural numbers themselves. For each distribution, we established explicit constructions for optimal sets and determined the associated quantization errors, sometimes for values of 
$n$ as large as 300.

Moreover, we addressed the inverse problem of recovering probability distributions from known optimal sets. Our study included identifying the conditions under which given sets form optimal quantizers and proposing conjectures supported by computational verification. These results contribute to the broader understanding of quantization in discrete settings, offering both theoretical insights and practical tools applicable in areas such as data compression and signal processing.

Future work may explore whether the conjectured optimal sets continue to hold beyond the computed range and further generalize the reverse problem to more complex classes of distributions.


\section*{Declaration}
							
							\noindent
\textbf{Authors' contributions:} Each author contributed equally to this manuscript. All authors have read and agreed to the published version of the manuscript.\\
\\
\noindent \tbf{Funding:}  This research received no external funding.\\
							
							\noindent
							\textbf{Data availability:} No data were used to support this study.\\
							\\
							\noindent
\textbf{Conflicts of interest.} The authors declare no conflict of interest.\\

\end{document}